\documentclass[12pt, english, a4paper]{amsart}
\usepackage[margin=2.5cm]{geometry}
\usepackage{amssymb,amsmath,amsthm}
\usepackage{commath}
\usepackage{mathtools}
\usepackage{mathrsfs}
\usepackage{accents}
\usepackage[utf8]{inputenc}

\usepackage{caption}

\usepackage{enumerate}
\usepackage{graphics, float}

\usepackage[colorlinks]{hyperref}
\hypersetup{allcolors = black}

\usepackage{url}
\usepackage[T1]{fontenc} 
\usepackage{fourier}
\usepackage{bm}
\usepackage{algorithm}
\usepackage{algpseudocode}
\usepackage{datetime}

\theoremstyle{plain}
\newtheorem{theorem}{Theorem}[section]

\newtheorem{prop}[theorem]{Proposition}
\newtheorem{corollary}[theorem]{Corollary}
\newtheorem{lemma}[theorem]{Lemma}
\newtheorem{definition}[theorem]{Definition}
\newtheorem{question}[theorem]{Question}

\newcommand{\N}{\mathcal{N}}
\newcommand{\R}{\mathbb{R}}
\newcommand{\conv}{\mathrm{conv}\,}
\newcommand{\aff}{\mathrm{aff}\,}
\newcommand{\lin}{\mathrm{lin}\,}

\newcommand{\eps}{\varepsilon}
\newcommand{\PP}{\mathbb{P}}
\newcommand{\EE}{\mathbb{E}}

\newcommand{\vb}{\mathrm{vb}}

\makeatletter
\@namedef{subjclassname@2020}{
  \textup{2020} Mathematics Subject Classification}
\makeatother

\usepackage[displaymath,textmath,graphics, subfigure, floats]{preview}
\PreviewEnvironment{center} 
\PreviewEnvironment{pspicture} 

\makeatother

\title{Colorful Vector Balancing}
\author{Gergely Ambrus}

\author{Rainie Bozzai}
\thanks{
Research of G. A. was partially supported by ERC Advanced Grant "GeoScape",  by the Hungarian National Research grant no. NKFIH KKP-133819,  and by project no. TKP2021-NVA-09. Project no.
TKP2021-NVA-09 has been implemented with the support provided by the
Ministry of Innovation and Technology of Hungary from the National
Research, Development and Innovation Fund, financed under the
TKP2021-NVA funding scheme. Research of R. B. is supported by the National Science Foundation Graduate Research Fellowship under Grant No. DGE-2140004 and Grant No. G20221001-4371 in Aid of Research from Sigma Xi, The Scientific Research Honor Society.}

\begin{document}

\begin{abstract}
We extend  classical estimates for the vector balancing constant of $\R^d$ equipped with the Euclidean and the maximum norms proved in the 1980's by showing that for $p =2$ and $p=\infty$, given vector families $V_1, \ldots, V_n \subset B_p^d$ with $0 \in \sum_{i=1}^n \conv V_i$, one may select vectors $v_i \in V_i$ with 
\[
\| v_1 + \ldots + v_n \|_2 \leq \sqrt{d}
\]
for $p=2$, and 
\[
\| v_1 + \ldots + v_n \|_\infty \leq O(\sqrt{d})
\]
for $p = \infty$. These bounds are sharp and asymptotically sharp, respectively,  for $n \geq d$.
The proofs combine linear algebraic and probabilistic methods with a Gaussian random walk argument.
\end{abstract}



\maketitle

\section{History and results}\label{intro}

Vector balancing problems have been studied for over six decades. The prototype question was asked by Dvoretzky~\cite{D63} in 1963:

\begin{question}[Vector balancing in $\ell_p$]\label{question_dvoretzky}
Let $1 \leq p$ and $1 \leq d \leq n$. Find 
\[
\max \min \|\eps_1 v_1 + \ldots + \eps_n v_n \|_p 
\]
where the maximum is taken over $v_1, \ldots, v_n \in B_p^d$, and the minimum is taken over
sign sequences $\eps_i \in \{ \pm 1 \}$ for $i = 1, \ldots, n$.
\end{question}

\noindent
In other words, given a collection of vectors $\{v_1, \ldots, v_n\} \subset B_p^d$, the goal is to find a signed sum of them with small norm.  (As usual, $B_p^d$ stands for the unit ball of the $\ell_p$-norm on $\R^d$).
The term ``vector balancing'' is readily motivated by the following interpretation: placing the vectors into the two plates of a scale according to their associated signs, the problem asks for achieving a nearly equal balance, that is, forcing the sum of the vectors in the plates to be as close as possible.

Even though the above formulation uses two parameters: $d$ and $n$, surprisingly, the optimal bounds turn out to depend only on the dimension $d$ \cite{BG81, Sp81, Sp85, RR22}. This was the first intriguing characteristic of the problem, with many more to follow in the subsequent decades.

In order to facilitate the forthcoming discussion,  we introduce the notion of {\em vector balancing constants} by defining
\begin{equation}
    \vb(K,L,n)=\max_{v_1, \ldots, v_n \in K} \min_{\eps_1, \ldots, \eps_n \in \{ \pm 1\}} \| \eps_1 v_1 + \ldots + \eps_n v_n \|_L
\end{equation}
for origin-symmetric convex bodies (i.e. compact, convex sets with nonempty interior) $K, L \subset \R^d$. Above, $\|.\|_L$ denotes the Minkowski norm associated to $L$, which is defined by 
\[
\|x\|_L = \min\{r\geq 0 : \ x \in r L \}
\]
for $x \in \R^d$ (note that $L$ is the unit ball of $\|.\|_L$). Then the vector balancing constant of $K$ and $L$ is given by
\begin{equation*}
    \vb(K)=\sup_{n \geq d} \vb(K,n).
\end{equation*}
When $K=L$, we  simply write 
\[
\vb(K,n) = \vb(K,K,n)
\]
and
\[
\vb(K)=\sup_{n \geq d} \vb(K,n).
\]
Thus, using the above notation, Question~\ref{question_dvoretzky} asks to determine $\vb (B_p^d,n)$ for $p \geq 1$ and $n \geq d$.

Sharp estimates for Dvoretzky's question in the Euclidean norm were proven in the late 1970's by Sevast'yanov~\cite{sevastyanov1980approximate}, independently by Bárány (unpublished at the time, for the proof, see \cite{Ba10}), and also, perhaps, by V.V. Grinberg. The proofs of Sevast'yanov and Bárány are of linear algebraic flavor. By the probabilistic method, Spencer~ \cite{Sp81} proved the same result, which reads as
\begin{equation}\label{vb_l2}
\vb(B_2^d, d) = \vb(B_2^d) = \sqrt{d}.
\end{equation}
In both approaches, the proof boils down to showing that any point of a parallelotope in $\R^d$ may be approximated by a vertex with Euclidean error at most $\sqrt{d}$; this is the direct predecessor of our Proposition~\ref{Euclidean}. 

The case of the $\ell_{\infty}$-norm was  solved by Spencer~\cite{Sp85} by showing 
\begin{equation}\label{vb_d_linfty}
\vb(B_\infty^d, d) \leq C \sqrt{d}
\end{equation}
and
\begin{equation}\label{vb_linfty}
\vb(B_\infty^d )\leq 2 C \sqrt{d}
\end{equation}
for a universal constant $C <6$. These estimates are asymptocially sharp, as demonstrated by vector systems consisting of vertices of $B_\infty^d$. The weaker bound of $O(\sqrt{d\ln d})$ can be shown by applying the probabilistic method, but removing the $\sqrt{\ln d}$ factor is not possible using that approach. The upper bound \eqref{vb_d_linfty} was also shown, independently, by Gluskin~\cite{Gl89}, who applied Minkowski's theorem on lattice points and an argument of Kashin~\cite{Ka85}. Indeed, these results rely on the parallelotope approximation in the maximum norm, which  precedes our Proposition~\ref{Infinity}. These vertex approximation results are close relatives of the Beck-Fiala ``integer-making'' theorems~\cite{BF81}.

Extending \eqref{vb_l2} and \eqref{vb_linfty}, Reis and Rothvoss~\cite{RR22} proved that there exists a universal constant $C'$ for which $\mathrm{vb}(B_p^d)\leq C'\sqrt{d}$ holds for each $2\leq p\leq \infty$. This bound,  up to constant factor, also matches the lower bound of Banaszczyk~\cite{Ba93} for general norms.

Our article is mainly motivated by the work of Bárány and Grinberg \cite{BG81}, who generalized Question~\ref{question_dvoretzky} to vector families which may also be interpreted as color classes. They proved the following general result. 
\begin{theorem}[Bárány, Grinberg~\cite{BG81}] \label{thm_BG}
Assume that $B^d \subset \R^d$ is an origin-symmetric convex body, and $V_1, \ldots, V_n \subseteq B^d$ are vector families so that $0\in \sum_{i\in[n]}\conv V_i$. Then there exists a selection of vectors $v_i\in V_i$ for $i\in[n]$ such that
    \begin{equation}\label{vbuppergen}
        \Big\|\sum_{i\in[n]}v_i\Big\|_{B^d}\leq d.
    \end{equation}
\end{theorem}
The original vector balancing problem is retrieved by setting $V_i = \{ \pm v_i \}$ for $i \in [n]$ (where, as usual, $[n] = \{ 1, \ldots, n \}$). 

Taking $B^d = B_1^d$, $n = d$, and $V_i = \{ \pm e_i \}$ for $i \in [n]$ shows that Theorem~\ref{thm_BG} is sharp. Yet, for specific norms, asymptotically stronger estimates may hold. In light of \eqref{vb_l2} and \eqref{vb_linfty}, it is plausible to conjecture that for the Euclidean and the maximum norms, the sharp estimate is of order $O(\sqrt{d})$. For the case of the Euclidean norm, it is mentioned in~\cite{BG81} that V. V. Grinberg proved the sharp bound of $\sqrt{d}$, although this has never been published (or verified) -- and 25 years later, the statement was again referred to as a conjecture~\cite{BD06}. Bárány and Grinberg~\cite{BG81} also note that ``from the point of view of applications, it would be interesting to know more about'' the case of the $\ell_\infty$-norm.  Based on a generalization of the Gram-Schmidt walk, Bansal et al. \cite{algban} recently proved the following theorem in the colorful setting.

\begin{theorem}[Bansal, Dadush, Garg, Lovett \cite{algban}]\label{algbanthm}
    Let $V_1,...,V_n \subseteq B_2^d$ be vector families with $0\in \mathrm{conv}V_i$ for each $i\in[n]$. Then for any convex body $K$ with $\gamma_d(K)\geq 1/2$, there exist vectors $v_i\in V_i$ such that $\sum_{i=1}^nv_i\in cK$, where $c>0$ is an absolute constant.
\end{theorem}

Applying this theorem to the Euclidean norm, one retrieves a sum of norm at most $C\sqrt{d}$ for some constant $C>1$, and for the maximum norm one obtains a bound of $O(\sqrt{d\ln d})$. We note that their proof method, which is based on the techniques of Lovász, Spencer, and Vesztergombi~\cite{LSV86}, can be modified to show that the bound in the colorful setting is at most twice the original vector balancing constant, which implies $O(\sqrt{d})$ bounds for both the Euclidean and maximum norm.  This asymptotically matches the estimates proved in the pre{\-}sent paper up to constants; for details, see Section \ref{BDGL}. Here, we provide a direct, constructive approach for proving an asymptotically matching, yet tighter estimate  which also sheds more light on the geometry of the problem and its algorithmic aspects.

We set off to obtain sharp/asymptotically sharp bounds for the colorful vector balancing problem in the Euclidean and maximum norm. Our proof techniques yield estimates matching \eqref{vb_l2} and \eqref{vb_linfty}, the bounds for the original vector balancing problem.

\begin{theorem}[Colorful Balancing in Euclidean Norm]\label{EuclideanThm}
    Given vector families $V_1,\dots,V_n\subseteq B_2^d$ with $0\in \sum_{i\in[n]}\conv V_i$, there exists a selection of vectors $v_i\in V_i$ for $i\in[n]$ such that
    \begin{equation*}
        \Big\|\sum_{i\in[n]}v_i\Big\|_2\leq \sqrt{d}.
    \end{equation*}
\end{theorem}
Note that the condition  $0\in \sum_{i\in[n]}\conv V_i$ is weaker than requiring $0 \in \conv V_i$ for each~$i$ -- by applying a shift of each family, the more general estimate can be derived from the statement under this more restrictive condition, albeit with the loss of a factor $2$ compared to the above bound.

For an estimate in the dual direction, the following result is well known: if $V_1,\dots,V_n$ are sets of {\em unit} vectors with $ 0 \in  \conv V_i$ for each $i$,  then one may select   $v_i\in V_i$ for $i\in[n]$  so that $ \Big\|\sum_{i\in[n]}v_i\Big\|_2\geq \sqrt{n}$ (for a further generalization, see~\cite{A23}).

\begin{theorem}[Colorful Balancing in Maximum Norm]\label{InfinityThm}
    Given vector families $V_1$,\dots,$V_n\subseteq B_\infty^d$ with $0\in \sum_{i\in[n]}\conv V_i$, there exists $v_i\in V_i$ for $i\in[n]$ such that
    \begin{equation*}
\Big\|\sum_{i\in[n]}v_i\Big\|_\infty\leq C\sqrt{d},
    \end{equation*}
    where $C=22$ suffices.
\end{theorem}

As noted above, an $O(\sqrt{d\ln d})$ estimate is implied by Theorem~\ref{algbanthm} -- we note that  a matching bound can also be derived by applying the standard probabilistic method.

By Carathéodory's theorem we may assume that each $V_i$ is finite (in fact, of cardinality at most $d+1$). We are going to use this assumption without further mention wherever needed.

Let us conclude this section with a short selection of related results. Many classical vector balancing results are surveyed by Giannopoulos~\cite{Gi97}. Vector balancing in the plane was studied by Swanepoel~\cite{Sw00} and Lund, Magazinov~\cite{LM17}. Online versions of vector balancing and related combinatorial games were considered by Spencer in \cite{Sp77} and \cite{Sp86}. Various anti-balancing questions were discussed by Banaszczyk~\cite{Ba93} and Ambrus, González Merino~\cite{AGM21}. 
Vector balancing for arbitrary convex bodies $K$ was connected to the Gaussian measure of $K$ by Banaszczyk~\cite{Ba98}. Several recent results have been given which provide algorithmic proofs for classical and online vector balancing problems, see for example Bansal~\cite{Bansal}, Bansal et al. \cite{BJSS20, algban},  Dadush et al. \cite{DNTT18}, Lovett and Meka~\cite{LM15}.
Additionally, vector balancing has recently been applied to various topics in machine learning, including sketches, coresets and randomized control trials, see Karnin, Liberty \cite{LibKar}, Harshaw et al. \cite{Harshaw2019BalancingCI}, among others.

Vector balancing in the maximum norm is also closely related to discrepancy theory, since  $\mathrm{vb}(B_\infty^d, n)$ may be interpreted as an upper bound for the discrepancy of a set system of cardinality $d$ on $n$ elements. For that direction, see e.g. Matou\v{s}ek~\cite{Ma99}.

One cannot miss to present a particularly attractive conjecture by Komlós (see \cite{Sp85,Ba98}), which states that
\[
\vb(B_2^d, B_\infty^d, n) \leq C
\]
holds for each $n, d \geq 1$ with a universal constant $C$. The strongest result in this direction is due to Banaszczyk~\cite{Ba98} who proved the upper bound of $O(\sqrt{\ln d})$. An algorithmic proof which yields a matching estimate was given recently by Bansal, Dadush and Garg~\cite{BDG19}.

Due to space limitations, other classical vector summation problems, e.g. the Steinitz lemma, questions about subset sums and partial sums, and the role of permutations are not addressed here. We only note that a colorful generalization of the Steinitz lemma has recently been given by Oertel, Paat and Weismantel~\cite{OPW22}. These directions are well portrayed in the book of Spencer~\cite{Sp94} and the survey articles of Bárány~\cite{Ba10} and Sevast'yanov~\cite{Se94, Se05}.

\section{Reducing to Dimension Many Families}\label{ReducingtoD}

Using the method of linear dependencies, we first prove that the number of families can be reduced from $n$ to  at most $d$, moreover, the total number of vectors can also be bounded from above. This approach  dates back to the classical work of Shapley and Folkman, and Starr in the 1960's~\cite{starr1969quasi}.
Several applications of the method  are well surveyed by Bárány~\cite{Ba10}.

We start with some terminology. 
Vectors in $\R^d$ will be 
written as $x = (x^{(1)}, \ldots, x^{(d)}) \in \R^d$ and will be understood as column vectors. 

As usual, $\{e_1, \ldots, e_d\}$ stands for the standard orthonormal basis of $\R^d$. 
We denote the standard $(m-1)$-dimensional simplex represented in $\R^m$ by 
\begin{equation*}
\Delta^{m}:=\conv \{ e_1, \ldots, e_m \} = \big\{\lambda\in \mathbb{R}^m:\ \lambda^{(i)}\geq 0\ \forall i\in[m],\ \sum_{i\in[m]}\lambda^{(i)}=1\big\}.
\end{equation*} In this paper, a convex polytope means a non-empty, bounded intersection of finitely many closed halfspaces (without any requirement on its interior).

 We identify a set of vectors $U = \{u_1, \ldots, u_m\} \subset \R^d$ with the $d \times m$ matrix 
\[U =
\begin{pmatrix}
u_1 \cdots u_m
\end{pmatrix}.
\]

\begin{definition}\label{matrix}
Given 
vector families $V_1,\dots,V_n \subset \R^d$ with $|V_i|=m_i$ and $\sum_{i\in[n]}m_i=m$,
we define the associated {\em vector family matrix}
\begin{equation*}
        V:=\begin{pmatrix}
V_1 | V_2 | \cdots | V_n
\end{pmatrix}        
 \in \R^{d\times m},
\end{equation*}
which is a partitioned matrix.  
We also introduce the  associated {\em set of convex coefficients }
\begin{equation}\label{Deltadef}
\Delta_V:=\Delta^{m_1}\times\cdots\times \Delta^{m_n} \subset \R^m
\end{equation}
which is a convex polytope arising as a direct product of simplices. 
\end{definition}

The relevance of $\Delta_V$ is shown by the fact that a vector $v \in \R^d$ is contained in $  \conv  V_i$ if and only if $v = V_i \lambda$ for some $\lambda \in \Delta^{m_i}$. Accordingly,
\[
\conv V_1+\cdots+\conv V_n = \{ V \lambda: \ \lambda \in \Delta_V \}.
\]

In the above scenario, we will usually consider $\R^m $ along with its orthogonal decomposition $\R^m=\R^{m_1}\times\cdots\times \R^{m_n}$. A collection of vector families $V = \{ V_1, \ldots, V_n \}$ will always be identified with its associated vector family matrix $V$ -- using the same notation for these two will cause no ambiguity and will be clarified by the context. From now on, $U, V$ and $W$  will always stand for a collection of vector families or their associated vector family matrices.

Throughout the paper, Greek letters will be used to denote vectors in the coefficient space $\Delta_V\subset\R^m$, while letters of the Latin alphabet will stand for vectors in $\mathbb{R}^d$. To make the connection between 
these spaces explicit, coefficient vectors $\beta\in \Delta_V$ will also be indexed by members of $V_i$ as follows:
\begin{equation}\label{betav}
    \beta=(\beta(v_i))_{v_i\in V_i,\ i\in[n]} \in \R^m. 
\end{equation}

Given a vector family matrix $V\in\R^{m\times d}$ and a set of indices $J\subset [m]$, we naturally define $V|_J$, the restriction of $V$ to the columns indexed by elements of $J$. This is again a vector family matrix which naturally induces a collection of vector families, the restrictions of the original ones to $J$. Naturally, $\Delta_{V|_J}\subset \R^{|J|}$ is the set of convex coefficients associated to $V|_J$. By virtue of the indexing \eqref{betav}, we may also define the restriction to a subcollection $W \subset V$. In particular, for $\beta \in \Delta_V$ and $V_i \in V$, $\beta|_{V_i} \in \Delta_{V_i}$ consists of the coefficients of vectors in $V_i$.

Given a partition $I\dot{\cup}J=[m]$ and vectors $\lambda\in \Delta_{V|_I}$, $\mu\in\Delta_{V|_J}$, we introduce the natural concatenation of $\lambda$ and $\mu$ by $\lambda\vee\mu\in \Delta_V$; that is, $(\lambda\vee \mu)|_{I}=\lambda$ and $(\lambda\vee\mu)|_J=\mu$. 

\begin{definition}
A number $x \in [0,1]$ is {\em fractional} if $x \not \in \{ 0,1 \}$. 
    Given a vector $\beta\in \Delta_V$, we say that family $V_i$ is {\em locked by $\beta$} if none of the coordinates of $\beta|_{V_i}$ are fractional. 
    Otherwise family $V_i$ is {\em free under~$\beta$}. A vector $\beta\in \Delta_V$ is a {\em selection vector} if every family is locked by $\beta$, equivalently, $\beta$ is a vertex of~$\Delta_V$. 
 \end{definition}

\noindent Note that for a selection vector $\beta \in \Delta_V$, ${V_i \beta|_{V_i} = v_i}$ for some $v_i \in V_i$  for each $i\in[n]$.

The main tool of the section   is the following generalization of the Shapley-Folkman lemma \cite{starr1969quasi, schneider_2013}, a cornerstone result in econometric theory.  Alternative versions were  proved and used by Grinberg and Sevast'yanov \cite{GS80} and Bárány and Grinberg \cite{BG81}.

\begin{theorem}\label{reduction}
Given a collection of vector families $V =\{V_1,\dots,V_n\}$ in $\R^d$ with $0\in \sum_{i\in[n]}\conv  V_i$,
there exists a vector $\alpha\in\Delta_V$ such that
\begin{itemize}
    \item[(i)] $V\alpha=0$;
    \item[(ii)] All but $k\leq d$ families $V_i$ are locked by $\alpha$;
    \item[(iii)]$\alpha$ has at most $k+d$ fractional coordinates.
\end{itemize}
\end{theorem}
The proof is based on the Shapley-Folkman--style statement below which is related to the geometry of basic feasible solutions of linear programs. 
\begin{lemma}[~\cite{Starr},~\cite{GS80}]\label{tight}
Let $K$ be a polyhedron in $\mathbb{R}^m$ defined by a system
\begin{equation*}
\left\{ \begin{aligned} 
 f_i(x)&=a_i,\ i=1,\dots,p, \\
    g_j(x)&\leq b_j,\ j=1,\dots,q,
\end{aligned} \right.
\end{equation*}
where $f_i,g_j$ are linear functions. Let $x_0$ be a vertex of $K$ and $A=\{j:g_j(x_0)=b_j\}$. Then $|A|\geq m-p$.
\end{lemma}

\begin{proof}[Proof of Theorem~\ref{reduction}]
     Given vector families $V_1,\dots,V_n$ in $\mathbb{R}^d$ with $0\in \sum_{i\in[n]}\conv V_i$ and $m = \sum_{i \in [n]} |V_i|$, consider the set 
\begin{equation} \label{P_def}
\begin{split}
P&=\big \{ \lambda \in \Delta_V: \ V \lambda = 0 \big \}   \\
&=\Big \{\lambda\in \mathbb{R}^m:\ \sum_{i\in[n]}\sum_{v_i\in V_i}\lambda(v_i)v_i=0,\ \ \sum_{v_i\in V_i}{\lambda(v_i)}=1\ \forall i\in[n],\ \ \lambda(v_i)\geq 0\ \forall i\in[n], \forall v_i\in V_i \Big \}.
\end{split}
\end{equation}
By our assumption that $0\in \sum_{i\in[n]}\conv V_i$, $P$ is a (non-empty) convex polytope in $\R^m$. Let $\alpha\in P$ be any extreme point of $P$.
Define
\begin{equation*}
    S:=\big \{i\in[n]:\ V_i \mbox{ is free under } \alpha \big  \}
\end{equation*}
and let $k = |S|$. By Lemma~\ref{tight}, at most $n+d$ non-negativity inequalities in
\eqref{P_def} are slack when substituting $\lambda = \alpha$. Each of the $n-k$ families locked by $\alpha$ contribute exactly one slack constraint, arising from the (unique) 1-coordinate. Let $f$ denote the number of fractional coordinates of $\alpha$; then $f + (n-k)$ is the total number of slack constraints. Thus
\[
f + (n -k) \leq n+d
\]
which implies that $f \leq k + d$. Since, by definition, $f \geq 2 k$, this also shows that $k \leq d$.
\end{proof}

By virtue of allowing  to reduce consideration to at most $d$ families, the following corollary is the main tool for proving upper bounds for the colorful vector balancing problem in arbitrary norms.

\begin{corollary}\label{coro}
Let $\|\cdot\|$ be a norm  on $\mathbb{R}^d$ with unit ball $B^d$. Suppose there exists a constant $C(d)$ such that given any collection of $k\leq d$ families $U = \{ U_1,\dots,U_k \} $ in $ B^d$ satisfying $|U_1|+\cdots+|U_k|\leq k+d$, and any $\lambda \in \Delta_U$, there exists a selection vector $\mu\in \Delta_U$ such that 
\begin{equation*}
\|V\lambda-V\mu\|\leq C(d).
\end{equation*}
Then given any collection of families $V_1,\dots,V_n\subseteq B^d$ with $0\in\sum_{i\in[n]}\conv V_i$, there exists a selection of vectors $v_i\in V_i$ for $i\in[n]$ such that
\begin{equation*}
    \Big\|\sum_{i\in[n]}v_i\Big\|\leq C(d).
\end{equation*}
\end{corollary}

\begin{proof}
    Suppose that the hypothesis of the statement holds. Let $m:=|V_1|+\cdots+|V_n|$. Applying Theorem~\ref{reduction} to $V = \{V_1,\dots,V_n\}$, we find $\alpha\in \Delta_V$ such that $V\alpha=0$, all but $k\leq d$ families $V_i$ are locked by $\alpha$, and $\alpha$ has at most $k+d$ fractional coordinates. Let $F\subset [m]$ be the set of indices of fractional coordinates, and set $L = [m] \setminus F$. Then $|F|\leq k+d$. By hypothesis, there exists a selection vector $\mu\in \Delta_{V|_F}$ such that $\|V|_F\alpha|_{F}-V|_F\mu\|\leq C(d)$, and so
    \begin{equation*}
        \|V|_L\alpha|_L+V|_F\mu\|\leq \|V|_L\alpha|_{L}+V|_F\alpha|_{F}\|+\|-V|_F\alpha|_{F}+V|_F\mu\|\leq \| V \alpha \|+C(d)=C(d).
    \end{equation*}
Taking the selection of vectors given by $\alpha|_L\vee\mu$ completes the proof.
\end{proof}

\section{The Case of the Euclidean Norm}\label{EuclideanProof}

To prove Theorem~\ref{EuclideanThm}, it remains to prove the vertex approximation property for color classes, which generalizes the Lemma in~\cite{Sp81}, see also Theorem 4.1 of~\cite{Ba10}.

\begin{prop}[Colorful vertex approximation in Euclidean norm]\label{Euclidean}

    Given a collection of $k$ vector families $U=\{U_1,\dots,U_k\}$ in $ B_2^d$ and any  point $\lambda\in \Delta_U$, there exists a selection vector $\mu\in \Delta_U$ such that $\|U\lambda-U\mu\|_2\leq \sqrt{k}$.
\end{prop}

\medskip

Our proof is inspired by Spencer's argument for the vector balancing case \cite{Sp81}, and it works in any finite dimensional Hilbert space.
\begin{proof}
Define $x:=U\lambda\in \conv U_1+\cdots+\conv U_k$, so that
\begin{equation*}
    x=x_1+\cdots+x_k,\ \  x_i=\sum_{u_i\in U_i}\lambda(u_i)u_i\ \forall i\in[k],
\end{equation*}
where $\lambda|_{U_i} \in \Delta_{U_i}$ for each $i\in[k]$.
We define a vector-valued random variable $w_i\in\mathbb{R}^d$ for each $i\in[k]$, which takes the value $u_i$ with probability $\lambda(u_i)$ for each $u_i\in U_i$, independently of the other $w_j$'s, $j\in [k]\setminus\{i\}$. Then
\begin{equation*}
    \mathbb{E}[w_1+\cdots+w_k]=\sum_{i\in[k]}\mathbb{E}[w_i]=\sum_{i\in[k]}\sum_{u_i\in U_i}\lambda(u_i)U_i=\sum_{i\in[k]}x_i=x.
\end{equation*}
Component-wise this yields
\begin{equation}\label{eq0}
    \mathbb{E}\big [w_1^{(\ell)}+\cdots+w_k^{(\ell)}-x^{(\ell)} \big ]=0,\ \ \ell\in[d].
\end{equation}
For each $\ell\in[d]$, \eqref{eq0} and the independence of the $w_i$'s imply
\begin{align}\label{var}
\begin{split}
    \mathbb{E} \big[(w_1^{(\ell)}+\cdots+w_k^{(\ell)}-x^{(\ell)})^2\big]
    &=\mathbb{E}\big[(w_1^{(\ell)}+\cdots+w_k^{(\ell)}-x^{(\ell)})^2\big]-\mathbb{E}\big [w_{1}^{(\ell)}+\cdots+w_{k}^{(\ell)}-x^{(\ell)} \big ]^2\\
    &=\mathrm{Var} \big [w_{1}^{(\ell)}+\cdots+w_{k}^{(\ell)}-x^{(\ell)} \big ]\\
    &=\sum_{i\in[k]}\mathrm{Var}\big[w_{i}^{(\ell)}\big].
    \end{split}
\end{align}
Since
\begin{equation*}
    \|w_{1}+\cdots+w_{k}-x\|_2^2=\sum_{\ell=1}^d \big((w_{1}^{(\ell)}+\cdots+w_{k}^{(\ell)})-x^{(\ell)}\big)^2,
\end{equation*}
by linearity of expectation and \eqref{var} we conclude
\begin{align}\label{final}
\begin{split}
    \mathbb{E} \big[\|w_1+\cdots+w_{k}-x\|^2 \big]&=\sum_{\ell\in[d]}\mathbb{E}\big[(w_{1}^{(\ell)}+\cdots+w_{k}^{(\ell)}-x^{(\ell)})^2\big]\\
    &=\sum_{\ell\in[d]}\sum_{i\in[k]}\mathrm{Var} \big[w_{i}^{(\ell)}\big]\\
    &=\sum_{\ell\in[d]}\sum_{i\in[k]}\mathbb{E}\big[(w_{i}^{(\ell)})^2\big]-\sum_{\ell\in[d]}\sum_{i\in[k]}\mathbb{E}\big[w_{i}^{(\ell)}\big]^2\\
    &=\sum_{i\in[k]}\mathbb{E}\big[\|w_{i}\|_2^2\big]-\sum_{\ell\in[d]}\sum_{i\in[k]}\mathbb{E}\big[w_{i}^{(\ell)}\big]^2.
    \end{split}
\end{align}
Finally, we note that
\begin{equation*}
    \mathbb{E}\big[\|w_{i}\|_2^2\big]=\sum_{u_i\in U_i}\lambda(u_i)\cdot\|u_i\|_2^2\leq \sum_{u_i\in U_i}\lambda(u_i)=1,
\end{equation*}
hence continuing calculation \eqref{final},
\begin{equation*}
   \mathbb{E}\big[\|w_1+\cdots+w_{k}-x\|^2\big] \leq k-\sum_{\ell\in[d]}\sum_{i\in[k]}\mathbb{E}\big[w_{i}^{(\ell)}\big]^2\leq k.
\end{equation*}
It follows that for some specific choice of $u_i\in U_i$, $i\in[k]$, we have
\begin{equation*}
    \|u_{1}+\cdots+u_{k}-x\|_2^2\leq k.
\end{equation*}
The corresponding selection vector $\mu\in \Delta_U$ satisfies the proposition.
\end{proof}

Theorem~\ref{EuclideanThm} now follows immediately from Corollary~\ref{coro} and Proposition~\ref{Euclidean}.

\section{The Case of the Maximum Norm}\label{InfinityProof}
To prove Theorem~\ref{InfinityThm}, we need to show that the vertex approximation property (the analogue of  Proposition~\ref{Euclidean}) holds for the maximum norm. This result for  the original vector balancing problem is due to Spencer~\cite{Sp85} and, independently, Gluskin~\cite{Gl89}. As in the original vector balancing problem, the challenge is to remove the $\sqrt{\ln d}$ factor. Note that, unlike in the Euclidean case, we need to set an upper bound on the total cardinality of the vector systems. 

\begin{prop}[Colorful vertex approximation in Maximum norm]\label{Infinity}
Given a collection of $k$ vector families $U=\{U_1,\dots,U_k\}$ in $B_\infty^d$ satisfying $m:=|U_1|+\cdots+|U_k|\leq 2d$, and an arbitrary point $\lambda\in\Delta_U$, there exists a selection vector $\mu\in \Delta_U$ such that $\|U\lambda-U\mu\|_\infty\leq C\sqrt{d}$ for a universal constant $C>0$.
\end{prop}

Note that applying the probabilistic method directly, by mimicking the second proof of Proposition~\ref{Euclidean} results in the weaker upper bound of $O(\sqrt{d} \sqrt{\ln d})$. Thus, in order to reach the bound of $O(\sqrt{d})$, one must apply an alternative argument, just as in the case of the original vector balancing problem in the maximum norm. In \cite{Sp85}, Spencer utilized a partial coloring method in order to overcome this difficulty. This, and the argument of Lovett and Meka~\cite{LM15} are the predecessors of our approach described below.

We will prove Proposition~\ref{Infinity} by iterating the following lemma, which is a close relative of the Partial Coloring Lemma in~\cite{LM15}. We call it the \textit{skeleton approximation lemma}, as it approximates a point in the set of convex coefficients $\Delta_W \subset \R^m$ by a point on the $(m/2)$-skeleton of  $\Delta_W$.

\begin{lemma}[Skeleton Approximation]\label{partial}
Let $W=\{W_1,\dots,W_k\}$  be a collection of vector families in $B_\infty^d$  with $|W_i| \geq 2$   for each $i$ and $m:=|W_1|+\cdots+|W_k|\leq 2d$. Then for any point $\lambda\in \Delta_W$, there exists $\mu \in \Delta_W$ such that
\begin{itemize}
    \item[(i)] $\|W\lambda-W\mu\|_\infty\leq \eta \sqrt{m\ln\frac{\xi d}{m}}$ where $\eta, \xi$ are constants specified as 
    \begin{equation}\label{etaxi}
    \eta = \frac 7 3 \text{ and }\xi = 18;
    \end{equation}
    \item[(ii)] $\mu^{(i)}=0$ for at least $m/2$ indices $i\in[m]$.
\end{itemize}
\end{lemma}
The proof of Lemma~\ref{partial} is postponed to Section~\ref{sect_skelproof}. Now we deduce Proposition~\ref{Infinity} assuming Lemma~\ref{partial}. 

\begin{proof}[Proof of Proposition~\ref{Infinity}] 
We may assume that $|U_i| \geq 2$ for each $i$, since any convex coefficient vector corresponding to a 1-element family is necessarily a selection vector. 

By an inductive process, we are going to define  points $\lambda(s) \in \Delta_U$, sets of indices  $F(s), L(s) \subset [m]$, and cardinalities $m(s)$ for $s = 0, 1, \ldots$ so that for a suitably large $S$, $\lambda(S)$ is a selection vector with the desired properties.

To initiate the recursive process, take $\lambda(0)=\lambda$, let $F(0) \subset [m]$ be the set of indices of fractional coordinates of $\lambda(0)$, and $L(0)=[m]\setminus F(0)$ be the set of indices of coordinates of $\lambda(0)$ equal to $0$ or $1$. Introducing $m(0)=|F(0)|$, we have $m(0)\leq m \leq 2d$.

Assuming that iterative step $s$ has been taken, we define step number $s+1$ as follows. Apply Lemma~\ref{partial} to the vector family matrix $U(s) := U|_{F(s)}$ of total cardinality $m(s) \leq m$ and the point $\lambda(s)|_{F(s)} \in \Delta_{U(s)}$ to find $\mu(s+1) \in \Delta_{U(s) }$ with the prescribed properties. Define $\lambda(s+1) = \lambda(s)|_{L(s)} \vee \mu(s+1)$ to be natural concatenation of these two vectors, obtained by replacing the fractional coordinates of $ \lambda(s)$ by the approximating vector $\mu(s+1)$. Let $F(s+1) \subset [m]$ be the set of indices of fractional coordinates of $\lambda(s+1)$, $L(s+1)=[m]\setminus F(s+1)$, and set $m(s+1) = |F(s+1)|$.

By Property (i) of Lemma~\ref{partial} and the definition of $\lambda(s+1)$, for each $s \geq 0$
\begin{equation}\label{touse}
   \|U\lambda(s)-U\lambda(s+1)\|_\infty\leq \eta \sqrt{m(s)\ln\tfrac{\xi d}{m(s)}} .
 \end{equation}
Also, by property (ii) of Lemma~\ref{partial} we have that 
\begin{equation}
\label{m(s)}
m(s) \leq \frac{m}{2^s} \leq \frac{d}{2^{s -1}} \,.
\end{equation}
Since $m(s) \in \mathbb{N}$, this also yields that after a finite number $S$ of steps, $m(S) = 0$ will hold. Set $\mu = \lambda(S)$. We will show that $\mu$ fulfills the criteria of Proposition~\ref{Infinity}.

That $\mu \in \Delta_U$ is a selection vector is shown by $m(S) = 0$. 
To show the approximation property, note that the function $f(x) =x \ln (1/x)$ is increasing on the interval $[0, 1/4]$. Combined with  \eqref{touse},\eqref{m(s)}, and \eqref{etaxi}, this yields that
\begin{align}\label{UlUm}
\|U\lambda-U\mu\|_\infty&\leq \sum_{s=0}^{S-1}\|U \lambda(s)-U \lambda(s+1)\|_\infty\nonumber\\ 
&\leq \sum_{s=0}^{S-1} \eta \sqrt{\ln\frac{\xi d}{m(s)}}\sqrt{m(s)}\nonumber\\
&\leq \sum_{s=0}^{S-1}  \eta\sqrt{\ln\frac{\xi d}{ d/ 2^{s-1}}}\sqrt{\frac{d}{2^{s-1}}}\\
&\leq  \eta\sqrt{d} \sum_{s=0}^\infty {2^{-(s-1)/2}}{\sqrt{\ln(\xi) + \ln 2 \cdot (s-1)}}\nonumber\\
&<22\sqrt{d}.\nonumber
\qedhere
\end{align}
\end{proof}

As in the Euclidean case, Theorem~\ref{InfinityThm} now follows from Corollary~\ref{coro} and Proposition~\ref{Infinity}. A simple modification of the proof yields the following version of Proposition~\ref{Infinity}, which provides a significant strengthening for $m\ll2d$.
\begin{prop}
    Given a collection of vector families $U=\{U_1,\dots,U_k\}$ in $B_\infty^d$ such that $m=|U_1|+\cdots+|U_k|\leq 2d$ and any point $\lambda\in \Delta_U$, there exists a selection vector $\mu\in \Delta_U$ such that
    \begin{equation*}
    \|U\lambda-U\mu\|_\infty\leq K\sqrt{m}\sqrt{\ln\tfrac{18 d}{m}}
    \end{equation*} for a universal constant $K>0$.

\end{prop}

\begin{proof}
Take $\mu\in \Delta_V$ as in the proof of Proposition~\ref{Infinity}. Then, substituting \eqref{m(s)} in \eqref{UlUm},
\begin{align*}
\|U\lambda-U\mu\|_\infty&\leq \sum_{s=0}^{S-1}\eta\sqrt{\ln\frac{\xi d}{m/2^{s-1}}}\sqrt{\frac{m}{2^{s-1}}}\\
   &\leq  \eta\sqrt{m}\sum_{s=0}^\infty \frac{\sqrt{\ln \big(\tfrac{\xi d}{m}\big)+\ln 2 \cdot s}}{\sqrt{2^{s-1}}}\\
   &\leq \eta\sqrt{m}\sqrt{\ln\tfrac{\xi d}{m}}\sum_{s=0}^\infty \frac{\sqrt{1 + s }}{\sqrt{2^{s-1}}}\\
   &< 9 \eta\sqrt{m}\sqrt{\ln\tfrac{18 d}{m}} .
\qedhere
\end{align*}
\end{proof}

\section{Useful Properties of the Gaussian Distribution}\label{Prob}

Proving Lemma~\ref{partial}, the Skeleton Approximation Lemma, requires several standard facts about the behavior of Gaussian random variables. By $\N(\mu,\sigma^2)$ we denote the (1-dimensional) Gaussian distribution with mean $\mu$ and variance $\sigma^2$. Given a linear subspace $A\subseteq\R^d$, $\N(A)$ denotes the standard multi-dimensional Gaussian distribution on $A$, i.e. for $G\sim \N(A)$, $G=G_1a_1+\cdots +G_ma_m$, where $\{a_1,\dots,a_m\}$ is any orthonormal basis of $A$ and $G_1,\dots,G_m\sim \N(0,1)$ are independent Gaussian random variables (for further details, see \cite{Bill, Feller71}).

\begin{lemma}\label{variance}
    Let $A\subseteq \R^d$ be a linear subspace with $G\sim \N(A)$. Then given any $u\in \R^d$, $\langle G,u\rangle\sim \N(0,\sigma^2)$, with $\sigma^2 = \| P_A(u) \|^2 \leq \|u\|_2^2$, where $P_A(\cdot)$ denotes the orthogonal projection onto $A$. 
\end{lemma}

\begin{corollary}\label{ei}
    Let $A\subseteq \R^d$ be a linear subspace     
    with $G\sim \N(A)$ and define $\sigma_i$ by  $\langle G,e_i\rangle\sim \N(0,\sigma_i^2)$. Then $\sum_{i\in[d]}\sigma_i^2=\mathrm{dim}A$.
\end{corollary}

A proof of Lemma~\ref{variance} can be found in \cite[Section III.6 ]{Feller71} (see also \cite{LM15}).  These results are particularly useful when combined with the following standard tail estimate.

\begin{lemma}\label{upperdev} Given a Gaussian random variable $G\sim \N(\mu,\sigma^2)$, for all $t>0$,
\begin{equation*}
    \mathbb{P}\big[|G - \mu|\geq t\big]\leq \exp\big(-t^2/2\sigma^2\big).
\end{equation*}
\end{lemma}
 This result is a special case of the general version of Hoeffding's inequality (for a proof see e.g. \cite{vershynin}). We will also need a similar bound for \textit{martingales} with Gaussian steps. Recall that a sequence $\{X_i\}_{i\in \mathbb{N}}$ of real-valued random variables is a martingale if $\mathbb{E}[X_{n+1}|\ X_1,\dots,X_n]=X_n$.
\begin{lemma}[\cite{Bansal}]\label{martingale}
    Let $0=X_0,X_1,\dots,X_T$ be a martingale in $\R$ with steps $Y_i=X_i-X_{i-1}$ for $i\geq 1$. Suppose that for all $i\in[T]$, $Y_i|X_0,\dots,X_{i-1}$ is a Gaussian random variable with mean zero and variance at most $\sigma^2$. Then for any $c>0$,
    \begin{equation*}
        \mathbb{P}\big[|X_T|\geq \sigma c\sqrt{T}\big]\leq 2\exp(-c^2/2).
    \end{equation*}
    \end{lemma}
   Finally, we will need the following result about sequences of Gaussian random variables. This is a well-known result that can be found for example in \cite{vershynin}; we provide the standard proof for the reader's convenience.
    \begin{lemma}\label{steps}
    Let $X_i\sim \N(0,\sigma_i^2)$ with $\sigma_i\leq 1$ for $i=1,2,\dots$ be a sequence of not necessarily independent, jointly Gaussian random variables. Then for any $T\geq 2$,
    \begin{equation*}
    \mathbb{E}\max_{i\leq T}|X_i|\leq 6\sqrt{\ln T}
    \end{equation*}
    and
    \begin{equation*}
    \mathbb{E}\max_{i\leq T}|X_i|^2\leq 10\ln T.
    \end{equation*}
    \end{lemma}

\begin{proof}
We define the random variable $Y:=\max_{i\in \mathbb{N}}\tfrac{|X_i|}{\sqrt{1+\ln i}}$. Then by Lemma~\ref{upperdev}, the union bound, and the fact that $\sigma_i\leq 1$ for all $i$,
\begin{align}\label{calculation}
\begin{split}
\mathbb{E}[Y]&=\int_0^\infty \mathbb{P}[Y\geq y]dy\\
&=\int_0^2\mathbb{P}[Y\geq y]dy+\int_2^\infty \mathbb{P}[Y\geq y]dy\\
&\leq 2+\int_2^\infty \mathbb{P}\Big[\max_{i\in\mathbb{N}}\tfrac{|X_i|}{\sqrt{1+\ln i}}\geq y\Big]dy\\
&\leq 2+\int_2^\infty \sum_{i=1}^\infty \mathbb{P}\Big[|X_i|\geq y\sqrt{1+\ln i}\Big]dy\\
&\leq 2+\int_2^\infty \sum_{i=1}^\infty\exp\big(-y^2(1+\ln i)/2\sigma_i^2\big)dy\\
&\leq 2+\int_2^\infty \left(\sum_{i=1}^\infty i^{-y^2/2}\right)\exp(-y^2/2)dy\\
&\leq 2+\tfrac{\pi^2}{6}\cdot 0.06<3.
\end{split}
\end{align}
Finally, note that $\sqrt{1+\ln i}\leq \sqrt{1+\ln T}$ for all $i\in[T]$, hence the calculation in \eqref{calculation} yields
\begin{equation*}
\mathbb{E}\max_{i\leq T}\tfrac{|X_i|}{\sqrt{1+\ln T}}\leq \mathbb{E}\max_{i\leq T}\tfrac{|X_i|}{\sqrt{1+\ln i}}\leq \mathbb{E}\max_{i\in\mathbb{N}}\tfrac{|X_i|}{\sqrt{1+\ln i}}< 3.
\end{equation*}
Then for $T\geq 2$, $\mathbb{E}\max_{i\leq T}|X_i|< 3\sqrt{1+\ln T}\leq 6\sqrt{\ln T}.$ The proof for $\max_{i\leq T}|X_i|^2$ follows from an analogous calculation.
\end{proof}

\medskip

\section{Proof of the Skeleton Approximation Lemma}\label{sect_skelproof}
We complete the proof of Theorem~\ref{InfinityThm} by proving the crux of the argument, Lemma~\ref{partial}. This will be done by means of providing an algorithm that proves the following slightly weaker statement.

\begin{lemma}\label{alg}
Let $0.01>\delta>0$ be arbitrary, and let $W=\{W_1,\dots,W_k\}$ be a collection of vector families in $B_\infty^d$ which satisfies that $|W_i| \geq 2$ for each $i$, and $m:=\sum_{i\in[k]}|W_i|\leq 2d$.
Define 
\begin{equation}\label{omega}
\omega(m):=\eta\sqrt{m\ln\tfrac{\xi d}{m}}
\end{equation} 
where $\eta = \frac 7 3$ and $\xi =18$ as in \eqref{etaxi}.
Then for any $\gamma\in \Delta_W$  there exists $\hat{\gamma}\in \Delta_W$ such that
\begin{itemize}
    \item[(i)] $\|W\gamma-W\hat{\gamma}\|_\infty\leq \omega(m)$;
    \item[(ii)] $\hat{\gamma}^{(i)}\leq \delta$ for at least $m/2$ indices $i\in[m]$.
\end{itemize}
\end{lemma}

Lemma~\ref{partial} follows immediately from Lemma~\ref{alg} by standard compactness arguments.

\begin{proof}[Proof of Lemma~\ref{alg}]

For each $j \in [d]$, let  $W^j \in\mathbb{R}^m $ denote the $j$th row of the vector family matrix $W$. The condition $W \subset B_\infty^d$ ensures that $\|W^j\|_\infty \leq 1$ for each $j \in[d] $. Accordingly, 
\begin{equation}\label{Wj2}
\|W^j\|_2^2 \leq m
\end{equation}
for each $j$.

\begin{figure}[h]
    \centering
    \includegraphics[scale=0.4]{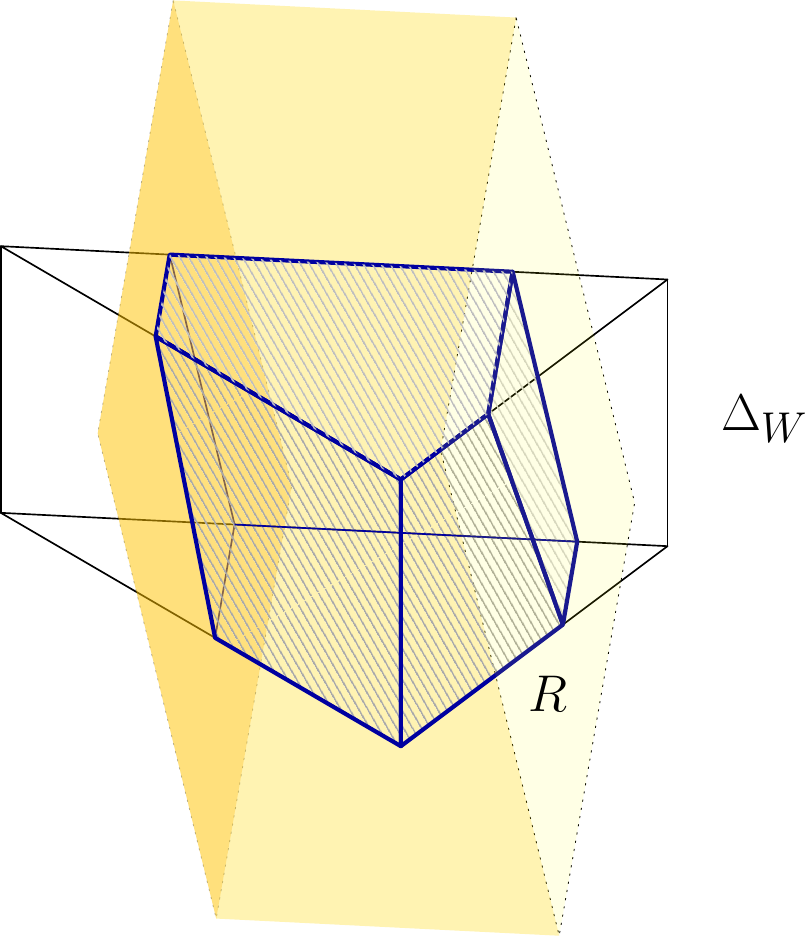}
    \caption{The polytope $R$}
    \label{R}
\end{figure}
\noindent
Consider the polytope 
\begin{equation*}
    R:=\Big\{\alpha\in \mathbb{R}^m:\ \alpha\in \Delta_W,\ \|W\alpha-W\gamma\|_\infty\leq \omega(m)\Big\},
\end{equation*}
which is the intersection of $\Delta_W$ with $d$ slabs of width $\omega(m)\big/\|W^j\|_\infty$, $j \in [d]$ (see Figure~\ref{R}). Equivalently, $R$ is defined by the  following set  of linear equations and inequalities: \begin{equation} \label{Rcons}
    R=\Big\{\alpha\in \mathbb{R}^{m}:\ \sum_{w_\ell\in W_\ell}\alpha(w_\ell)=1\ \forall \ell\in [k],\ \ \alpha^{(i)}\geq 0\ \forall i\in[m],\ |\langle \alpha-\gamma,W^j\rangle|\leq \omega(m)\ \ \forall j\in[d]\Big\}.
\end{equation}
We call the first and second set of constraints \textit{convexity constraints}, as they ensure that $W\alpha\in \conv(W_1)+\cdots+\conv(W_k)$ for each $\alpha \in R$. The third set of constraints will be referred to as \textit{maximum constraints}, as they imply that given $\alpha\in R$,
\begin{equation*}
  \|W\alpha-W\gamma\|_\infty=\max_{j\in[d]}|(W\alpha-W\gamma)_j|=\max_{j\in[d]} |\langle \alpha-\gamma, W^j\rangle|\leq \omega(m).
\end{equation*}
Let $Z$ be the set of normal vectors of the inequality constraints in \eqref{Rcons}:
\begin{equation}\label{Z}
    Z=\big \{e_1,\dots,e_m,W^1,\dots,W^d \big \}.
\end{equation}
By the previous remarks, $\| Z \|_\infty = 1$.
\medskip

The main tool of the argument is to introduce a suitable discrete time Gaussian random walk on $\R^m$, similar to that in \cite{Bansal} and  \cite{LM15}. In order to help the reader navigate through the forthcoming technical details, we first give an intuitive description of the walk, whose position at time $t = 0, 1, \dots$ will be denoted by $\Gamma_t \in \R^m$. 

The walk starts from $\Gamma_0 = \gamma$ and runs in $\aff R$ 
with sufficiently small Gaussian steps as long as $\Gamma_t$ is in the interior of $R$, far from its boundary. As $\Gamma_t$ gets $\delta$-close to crossing a facet of $R$,  we confine the walk to an affine subspace parallel to that facet for the subsequent steps, by intersecting the current range with a hyperplane parallel to the facet. In particular, if any coordinate of $\Gamma_t$ reaches a value less than $\delta$, we freeze that coordinate for the remainder of the walk. 

We show that running the walk long enough, until say time $T$,  with high probability at least half of the coordinates of $\Gamma_T$ become frozen, while $\Gamma_T \in R$ still holds. This will mean that $\hat{\gamma} = \Gamma_T$ satisfies the criteria of Lemma~\ref{alg}. For the proof it is essential that the value of $\omega(m)$ is carefully set \eqref{omega}, hence the slabs defining $R$ are sufficiently wide so that the walk is unlikely to escape from them. 

Let us turn to the formal definition of the random walk. Let $\varepsilon>0$ and $T\in\mathbb{N}$ be parameters to be defined later. 
Define the sets 
\begin{equation}\label{Ctdef}
C_t^{conv}:=\big\{i\in[m]:\ \Gamma_t^{(i)}\leq \delta\big\},\ \ \ \ \ \ C_t^{max}:=\big\{j\in[d]:\ |\langle\Gamma_{t}-\Gamma_0,W^j\rangle|\geq \omega(m)- \delta\big\}\end{equation} to be the convexity and maximum constraints, respectively, that are at most $\delta$-close to being violated by $\Gamma_t$. 
We will say that coordinate $i$ is {\em frozen} iff $i \in C_t^{conv}$. 
Recall that by \eqref{betav}, coordinates may be indexed by the vectors, that is, for each $i\in[m]$, $\Gamma_t^{(i)}=\Gamma_t(w_l)$ for some $l\in[k]$ and $w_l\in W_l$. In that case, coordinate $w_l$ is frozen iff $i \in C_t^{conv}$.

Let $A$ be the linear component of $\aff \Delta_W$, that is, $A=\lin(\Delta_W-\Delta_W)$. For each $t\leq 1$, step $t$ is confined to occur in the linear subspace 
\begin{equation*}
S_{t}:=\big\{\beta\in A:\ \beta^{(i)}=0\ \forall i\in C_{t-1}^{conv},\ \langle \beta-\Gamma_0, W^j\rangle=0\ \forall j\in C_{t-1}^{max}\big\}\end{equation*}
by taking a Gaussian step $\Lambda_t\sim \N(S_t)$ and defining
\[
\Gamma_{t}=\Gamma_{t-1}+\varepsilon \Lambda_t.
\]
The walk terminates after $T$ steps: $\hat{\gamma}:=\Gamma_T$, where $T$ is to be determined later.

\medskip

We will show that with certain restrictions on the parameters, $\Gamma_T$ satisfies properties (i) and (ii) of Lemma~\ref{alg}  with probability at least $0.2$.

Given $\eps>0$, we define
\begin{equation}\label{T}
T:= \left \lfloor \frac{0.99^2 \eta^2}{2 \eps^2} \right \rfloor.
\end{equation}
Choose $\eps>0$ small enough so that the following inequalities hold simultaneously:
\begin{equation}\label{6Td}
 6 T d \exp\Big( - \frac {\delta^2}{2 m \eps^2 }\Big) < 0.01,
\end{equation}
\begin{equation}\label{epsprob}
22\varepsilon m^2 \ln T \leq 0.01,
\end{equation}
and
\begin{equation}\label{B2bound}
  10 \eps^2 \ln T \leq 1.
\end{equation}
 This can indeed be guaranteed since the functions $\exp(-x)/x$, $x \ln \frac1{x^2}$ and $x \ln \frac1x$ converge to $0$ as $x \searrow 0$.

\medskip

We summarize a few useful properties of the random walk.

\begin{lemma}\label{props}
    Let $\Gamma_0,\dots,\Gamma_T$ be the steps of the Gaussian random walk defined above and $i\in[m]$, $j\in[d]$. Then:
    \begin{itemize}        
        \item[(i)]Given $\Gamma_{t-1}$, $\mathbb{E}[\Lambda_t]=0$.
        \item[(ii)]$C_t^{conv}, C_t^{max}$ are nested increasing sets in $t$.
        \item[(iii)] $S_t$ is a nested decreasing sequence of linear subspaces in $\R^m$ in $t$.
        \item[(iv)] At any time $0\leq t\leq T$ and for any $i\in[k]$, $\sum_{w_i\in W_i}\Gamma_t(w_i)=1$.
        \item[(v)] If the walk leaves the polytope $R$ at time $t\in[T]$, then $\Gamma_s \not \in R$ for any $s \geq t$.
        \item[(vi)] If the walk leaves the polytope $R$ at time $t\in[T]$, then $| \langle \Lambda_t, z\rangle | \geq \delta/\varepsilon$ for some $z\in Z$.
        \item[(vii)] If coordinate $i$ is frozen at step $t$, that is $i\in C_t^{conv}\setminus C^{conv}_{t-1}$, then $\Gamma_T^{(i)} = \Gamma_t^{(i)}\geq \delta-\eps \big|\Lambda_t^{(i)}\big|$.        
    \end{itemize}
\end{lemma}
\begin{proof}
Properties (i)-(v) are straightforward consequences of the definition of $\Gamma_t$.

    To prove (vi), suppose that the walk leaves the polytope $R$ at time $t$. Then an inequality constraint in \eqref{Rcons} with normal vector $z\in Z$ is violated at time $t$. Suppose that $z=W^j$ for some $j\in[d]$. Since  $j\not\in C^{max}_{t-1}$, 
     \[ | \langle \Gamma_{t-1}-\Gamma_0,W^j\rangle |  < \omega(m)-\delta,\]
    while on the other hand, 
    \[ | \langle \Gamma_{t}-\Gamma_0,W^j\rangle | > \omega(m). \]
    Combining these inequalities shows that 
    \begin{equation*}
   \varepsilon | \langle \Lambda_t,W^j\rangle | = | \langle \Gamma_t-\Gamma_{t-1},W^j\rangle | \geq \delta.
\end{equation*}
The proof when $z=e_i$ for $i\in[m]$ is analogous.

To prove (vii), note that $i\not\in C_{t-1}^{conv}$ implies that $\Gamma_{t-1}^{(i)}\geq \delta$. Therefore, 
    \[
\Gamma_T^{(i)}=\Gamma_t^{(i)} = \Gamma_{t-1}^{(i)}+\varepsilon\Lambda_t^{(i)} \geq \delta + \varepsilon\Lambda_t^{(i)} \geq \delta - \eps |\Lambda_t^{(i)}|.
\qedhere
    \]
\end{proof}

\medskip

Equipped with these properties, we are ready to prove that $\Gamma_T$ satisfies the required conditions of Lemma~\ref{alg} with probability at least $0.01$. To show (i), that is 
\begin{equation}\label{WGWT}
    \|W\Gamma_0-W\Gamma_T\|_\infty\leq \omega(m),
\end{equation}
it is sufficient to argue that (with high probability) $\Gamma_T\in R$; that is, the walk does not leave the polytope $R$ at any step.

Define the event $E_t:=\{\Gamma_t\not\in R|\ \Gamma_0,\dots,\Gamma_{t-1}\in R\}$ that the walk steps out of $R$ at time $t$. 
If $E_t$ occurs, then by Lemma~\ref{props}(vi), $|\langle \Lambda_t,z\rangle|\geq \delta/\varepsilon$ for some $z\in Z$.  By Lemma~\ref{variance} and \eqref{Wj2}, for any $z\in Z$, $\langle \Lambda_t,z\rangle$ is a Gaussian random variable with mean $0$ and variance $\sigma^2\leq m$. Applying Lemma~\ref{upperdev} to $\langle \Lambda_t,z\rangle$, we find
\begin{equation}\label{probbound}
    \PP\big[|\langle \Lambda_t,z\rangle| \geq \tfrac{\delta}{\varepsilon} \big]\leq 2\exp\big(-\big(\tfrac{\delta}{\varepsilon} \big)^2 /2m\big).
\end{equation}
Using the union bound, equations \eqref{Z}, \eqref{probbound},  and the fact that $d\leq 2d$, we derive
\begin{equation}
\label{firstbound}
\begin{split}
    \mathbb{P}[\exists t\in[T]:\ \Gamma_t\not\in R]&=\sum_{t=1}^T\mathbb{P}[E_t]\\
    &\leq \sum_{t=1}^T\sum_{z\in Z}\PP\big[|\langle \Lambda_t,z\rangle| \geq \tfrac{\delta}{\varepsilon} \big]\\
    &\leq 2T(d+m) \exp\big(-\big(\tfrac{\delta}{\varepsilon} \big)^2 /2m\big)\\
&\leq  6 T d \exp\Big( - \frac {\delta^2}{2 m \eps^2 }\Big)\\
&<0.01
\end{split}
\end{equation}
by condition~\eqref{6Td}. This proves that \eqref{WGWT} holds with probability at least $0.99$.
\medskip

\noindent
It remains to address (ii) of Lemma~\ref{alg}, that (with positive probability) $\Gamma_T^{(i)}\leq \delta$ for at least $m/2$ indices $i\in[m]$. We will reach this by means of proving that 
\begin{equation}\label{ECconv}
    \mathbb{E}[|C_T^{conv}|]>0.51m.
\end{equation}
To this end  we derive the following identity, using Lemma~\ref{props}(i):
\begin{align*}
    \mathbb{E} \big[\|\Gamma_t\|_2^2\big]&=\mathbb{E}\big[\|\Gamma_{t-1}+\varepsilon \Lambda_t\|_2^2\big]\\
    &=\mathbb{E}\big[\|\Gamma_{t-1}\|_2^2\big]+\varepsilon^2\mathbb{E}\big[\|\Lambda_t\|_2^2\big]+2\varepsilon\mathbb{E}\big[\langle \Gamma_{t-1}, \Lambda_t\rangle\big]\\
    &=\mathbb{E}\big[\|\Gamma_{t-1}\|_2^2\big]+\varepsilon^2\mathbb{E}\big[\mathrm{dim}(S_t)\big],
\end{align*}
where in the last equation we use that, by Corollary~\ref{ei},  
\[
\EE \big[ \|\Lambda_t\|_2^2 \big ]= \EE \big [\sum_{i\in[m]}\langle \Lambda_t,e_i\rangle ^2 \big ]=    \sum_{i\in[m]} \EE \big [ \langle \Lambda_t,e_i\rangle ^2 \big ]=\mathrm{dim}S_t.
\]
 Iterating this calculation and using Lemma~\ref{props}(iii),
\begin{equation*}
   \mathbb{E} \big[\|\Gamma_t\|_2^2\big]\geq \varepsilon^2\sum_{t\in[T]}\mathbb{E}\big[\mathrm{dim}(S_t)\big]\geq T\varepsilon^2\mathbb{E}\big[\mathrm{dim}S_T\big]=T \eps^2\,\mathbb{E}\big[m-|C_T^{conv}|-|C_T^{max}|\big],
\end{equation*}
and rearranging yields
\begin{equation}\label{keyineq}
\mathbb{E}\big[|C_T^{conv}|\big]\geq m-\frac{\mathbb{E}\big[\|\Gamma_T\|_2^2\big]}{T \eps^2}-\mathbb{E}\big[|C_T^{max}|\big].
\end{equation}
The above identity allows us to prove \eqref{ECconv} by giving upper estimates on $\mathbb{E}\big[\|\Gamma_T\|_2^2\big]$ and $\mathbb{E}\big[|C_T^{max}|\big]$.

We start with the second of these and show that 
\begin{equation}\label{ECmax}
\mathbb{E}\big[|C_T^{max}|\big] \leq \tfrac{2m}{\xi}.
\end{equation}
To this end, we bound the probability that the walk gets close to escaping from a given slab. 
Note that for fixed $j\in[d]$, $\{ \langle \Gamma_t-\Gamma_0, W^j\rangle\}_{t\in[T]}$ for $0\leq t\leq T$ is a martingale satisfying the conditions of Lemma~\ref{martingale}. As the step size is $\varepsilon\langle \Lambda_t, W^j\rangle$, by Lemma~\ref{variance} the variance of any step is bounded by $\eps^2 \|W^j\|_2^2 \leq \eps^2 m$ (cf. \eqref{Wj2}). 

For any $j\in C_T^{max}$,  by \eqref{Ctdef},
\begin{equation*}
    |\langle \Gamma_T-\Gamma_0,W^j\rangle|\geq \omega(m)-\delta\geq 0.99\,\omega(m),
\end{equation*} as we have $\delta\leq 0.01$ and $\omega(m)\geq 1$ by \eqref{omega}.

Therefore, by Lemma~\ref{martingale}, \eqref{omega}, and \eqref{T},
\begin{align*}
    \mathbb{P}[j\in C_T^{max}]&\leq \mathbb{P}\left[|\langle \Gamma_T-\Gamma_0,W^j\rangle|\geq 0.99 \, \omega(m)\right]\\
    &\leq 2\exp\Big(\frac{-0.99^2\cdot \eta^2 \, \ln (\xi d/m)}{2 T \varepsilon^2 }\Big)\\
    &<2\exp\big(\ln\tfrac{m}{\xi d}\big)=\frac{2m}{\xi d}.
\end{align*}
Thus
\begin{equation*}\label{discconst}
    \mathbb{E}\big[|C_T^{max}|\big]=\sum_{j\in [d]}\mathbb{P}\big[j\in C_T^{max}\big]<\tfrac{2m}{\xi}
\end{equation*}
as desired.
\medskip

To complete the proof of \eqref{ECconv} we address the second term in \eqref{keyineq} and show that
\begin{equation}\label{upperbound}
\mathbb{E}\big[\|\Gamma_T\|_2^2 \big]\leq 1.01 m.
\end{equation}
By \eqref{betav}, we represent $\Gamma_T$ in terms of the vector families as $\Gamma_T=\big(\Gamma_T(w_i)\big)$, ${w_i\in W_i,\ i\in[k]}$. Then
\begin{equation}\label{gammanormsquare}
\|\Gamma_T\|_2^2=\sum_{i\in[k]}\sum_{w_i\in W_i}\big(\Gamma_T(w_i)\big)^2.
\end{equation}
As the above double sum has $m$ terms in total, it suffices to show that the expectation of any of these terms is at most $1.01$, that is, $\mathbb{E}\big[\Gamma_T(w_i)^2\big]\leq 1.01$ for any $i\in[k]$ and $w_i\in W_i$.

By Lemma~\ref{props}(iv), $\sum_{w_i\in W_i} \Gamma_T(w_i) =1$.
Thus, for any fixed $w_i \in W_i$, 
\begin{equation}\label{Gamma_w_i}
    \Gamma_T(w_i)=1-\sum_{w\in W_i \setminus \{w_i\}}\Gamma_T(w).
\end{equation}
Note that in the above sum, $\Gamma_T(w)\geq 0$ unless coordinate $w$ is frozen. In this case, assuming that coordinate $w$ is frozen at step $t$, by Lemma~\ref{props}(vii) we have $ \Gamma_T(w) \geq  \delta - \eps \big|\Lambda_t(w)\big|$. Accordingly,  
\begin{equation}\label{gammatw}
\Gamma_T(w)\geq \delta-\max_{t\in[T]}\varepsilon\big|\Lambda_t(w)\big| > -\varepsilon\max_{t\in[T]}\big|\Lambda_t(w)\big|.
\end{equation}
Thus, by \eqref{Gamma_w_i}, 
\begin{equation*}
\Gamma_T(w_i)\leq 1+\varepsilon \sum_{w\in W_i \setminus \{w_i\}}\max_{t\in[T]}\big|\Lambda_t(w)\big|.
\end{equation*}
When $\Gamma_T(w_i) \geq 0$, this leads to 
\begin{equation*}
\begin{split}
\big(\Gamma_T(w_i) \big)^2\leq 1 &+2 \, \eps  \sum_{w\in W_i \setminus \{w_i\}}\max_{t\in[T]} \big|\Lambda_t(w)\big|+\varepsilon^2  \sum_{w\in W_i \setminus \{w_i\}}\max_{t\in[T]}\big|\Lambda_t(w)\big|^2 \\ &+\varepsilon^2  \sum_{w\neq u\in W_i \setminus 
\{w_i\}}\max_{t\in[T]}\big|\Lambda_t(w)\big|\max_{t\in[T]}\big|\Lambda_t(u)\big|.
\end{split}
\end{equation*}
Note that by Lemma~\ref{variance}, for each vector $w$, $\Lambda_t(w)$ is a Gaussian random variable with variance at most 1. Also, for $a, b \geq 0$ we will use that $ab \leq a^2 + b^2$. Therefore, by taking expectations above, and applying Lemma~\ref{steps}, 
\begin{align}
\label{uglycomp}
 \mathbb{E} \Big( \big( \Gamma_T(w_i) \big)^2 &\big| \, \Gamma_T(w_i) \geq 0 \Big)  \nonumber \\ &=
 1 +2 \, \eps  \sum_{w\in W_i \setminus \{w_i\}} \EE \big[ \max_{t\in[T]} \big|\Lambda_t(w)\big| \big]+\varepsilon^2  \sum_{w\in W_i \setminus \{w_i\}} \EE \big[ \max_{t\in[T]}\big|\Lambda_t(w)\big|^2 \big] \nonumber \\ & \hspace{1 cm}+\varepsilon^2  \sum_{w\neq u\in W_i \setminus 
\{w_i\}} \EE \big[ \max_{t\in[T]}\big|\Lambda_t(w)\big|\max_{t\in[T]}\big|\Lambda_t(u)\big| \big] \nonumber  \\
 &\leq 1+12\varepsilon m \sqrt{\ln T}+10\varepsilon^2 m \ln T + \varepsilon^2  \sum_{w\neq u\in W_i \setminus 
\{w_i\}}  \EE \big[ \max_{t\in[T]}\big|\Lambda_t(w)\big|^2 + \max_{t\in[T]}\big|\Lambda_t(u)\big|^2 \big] \\
&\leq 1+12\varepsilon m \sqrt{\ln T}+10\varepsilon^2 m \ln T + 2\tfrac{m(m-1)}{2}10 \eps^2 \ln T \nonumber  \\
&\leq 1+12\varepsilon m\sqrt{\ln T}+10\varepsilon^2m^2\ln T\nonumber\\
&\leq 1+22\varepsilon m^2 \ln T \nonumber \\
& \leq 1.01 \nonumber 
\end{align}
by \eqref{epsprob} and that $m,\ln T \geq 1, \varepsilon<1$.

When $\Gamma_T(w_i)< 0$, then coordinate $w_i$ is frozen. Therefore, \eqref{gammatw} and \eqref{B2bound} imply that 
\begin{align*}
    \mathbb{E} \Big( \big(\Gamma_T(w_i) \big)^2 \big| \,  \Gamma_T(w_i) < 0\Big) < \eps^2 \, \mathbb{E} \max_{t\in[T]} \big|\Lambda_t(w_i) \big|^2 \leq 10\varepsilon^2  \ln T<1.
\end{align*}
Combining this with \eqref{uglycomp} shows that $\mathbb{E}\left(\big(\Gamma_T(w_i)\big)^2 \right) \leq 1.01$ for each $i\in[k]$ and $w_i \in W_i$, and by invoking \eqref{gammanormsquare}, we reach \eqref{upperbound}.

\medskip

To prove \eqref{ECconv}, we may now combine \eqref{T}, \eqref{keyineq}, \eqref{ECmax} and  \eqref{upperbound} in order to derive that
\begin{align*}
\mathbb{E} \Big[ \big|C_T^{conv} \big| \Big] & \geq m-\frac{1.01m}{T \eps^2}-\frac{2m}{\xi}\\
&\geq \left(1 - \frac{2 \cdot 1.01}{0.99^2 \eta^2} -\frac{2}{\xi} \right) m\\
&> 0.51m.
\end{align*}
Since $\big|C_T^{conv}\big| \leq m$, this leads to
\[
\PP\Big[\big|C_T^{conv}\big|\geq m/2 \Big]\geq 0.02.
\]
As the probability  of the walk leaving $R$ is less than $0.01$ by \eqref{firstbound}, we conclude that the algorithm finds the desired vector $\Gamma_T$ with probability greater than $0.02-0.01=0.01$, as claimed.
\end{proof}
\medskip
Finally, we illustrate how to transform the proof of  Proposition~\ref{Infinity} so as to provide a polynomial time algorithm.

\begin{prop} There exists an algorithm of running time $O(d^7\ln^2 d)$ which, in the setting of Proposition~\ref{Infinity},  yields the desired selection vector $\mu\in \Delta_U$.
\end{prop}

\begin{proof}
Along the course of the  proof of Proposition~\ref{Infinity}, we replace the iteration of Lemma~\ref{partial} by that of 
Lemma~\ref{alg} so  as to obtain a vector $\hat{\mu}\in \Delta_U$ such that, for each $i\in[k]$, 
\begin{equation}\label{muhat}
|\{ w_i \in W_i: \ 0\leq\hat{\mu}(w_i)\leq \delta \} |  = |W_i|-1 .
\end{equation}
The existence of such a vector is guaranteed as long as $\delta<1/|W_i|$ for each $i \in [k]$. 
At the final step, we take $\mu$ to be the closest vertex of $\Delta_W$ to $\hat{\mu}$, that is, define
\begin{equation*}
\mu(w_i) = 
\begin{cases}
0 \textrm{ if } \hat{\mu}(w_i)\leq \delta \\
1 \textrm{ if } \hat{\mu}(w_i)> \delta.
\end{cases}
\end{equation*}
In particular,
taking $\chi:=\mu-\hat{\mu}$, we have that by \eqref{UlUm},
\begin{equation*}
    |\langle\mu-\lambda,W^j\rangle|=|\langle  \hat{\mu}-\lambda,W^j\rangle + \langle \chi,W^j\rangle|\leq 22\sqrt{d}+|\langle \chi,W^j\rangle|
\end{equation*}
 for each $j\in[d]$.
We show that taking a sufficiently small value of $\delta$ ensures that $|\langle \chi, W^j\rangle|\leq O(\sqrt{d})$, accordingly, $\mu$ is an appropriate selection vector.

Let $i \in [k]$ be arbitrary, and let $w \in W_i$ be so that $\hat{\mu}(w) > \delta$. Then  $\hat{\mu}(w)=1 - \sum_{w_i\neq w\in W_i}\hat{\mu}(w_i)$. Accordingly, 
\[
\sum_{w_i\in W_i}|\chi(w_i)| = 2 \sum_{w\neq w_i\in W_i}|\hat{\mu}(w_i)| < 2 |W_i| \, \delta.
\]
Therefore, as $\|W^j\|_\infty\leq 1$,
\begin{equation*}
    |\langle \chi,W^j\rangle|\leq \sum_{i\in[k]}\sum_{w_i\in W_i}|\chi(w_i)|\leq\sum_{i\in[k]}(2|W_i|\, \delta).
\end{equation*}
Since $|W_i| \leq m\leq 2d$ for each $i\in[k]$ and $k\leq m\leq 2d$, we conclude that for each $j\in[d]$, $|\langle \chi,W^j\rangle|\leq 8d^2\delta$. Thus, fixing 
\begin{equation}\label{deltadef}
    \delta=0.01 d^{-3/2},
\end{equation} we indeed obtain
\begin{equation*}
    \|W\mu-W\lambda\|_\infty=\max_{j\in [d]}|\langle \mu-\lambda, W^j\rangle|\leq 22\sqrt{d}+8\sqrt{d}=O(\sqrt{d}).
\end{equation*}
Next, we estimate the running time of the algorithm at a given iteration $s$ of Lemma~\ref{alg}. As before, let  $m(s)$ be the number of active vectors.
For a fixed step $t\in [T]$ of the Gaussian random walk, the calculation of the sets $C_{t}^{var}, C_t^{max}$ takes time $O(d+m(s))$. An orthonormal basis of the subspace $S_t$ may be determined in  $O(d^3)$ time by applying a Gram-Schmidt orthogonalization process, and then the Gaussian vector $\Lambda_{t+1}$ is sampled in $O(m(s))$ time. Hence, the time complexity of performing a given step of the Gaussian walk is dominated by the calculation of the orthonormal basis of~$S_t$. 

By \eqref{deltadef}, \eqref{m(s)}, and the condition that $m \leq 2d$, the maximal $\eps$ which satisfies the constraints \eqref{6Td}, \eqref{epsprob}, and \eqref{B2bound} simultaneously can be estimated by
\[
\frac{1}{\varepsilon^2} = O \big(m(s) d^3 \ln^2 d\big).
\]

Since, by \eqref{T}, the number of steps of the random walk is $T= O(1 / \eps^2)$, the above estimate shows that  the running time of the algorithm within a given iteration $s$ is
\[
O \big(m(s) d^6 \ln^2 d \big).
\]
Let $S$ be the total number of iterations until reaching the vector  $\hat{\mu}\in \Delta_U$ satisfying \eqref{muhat}. Note that by \eqref{m(s)}, we have $m(s) \leq m/2^s$. Therefore, the total running time of the algorithm is 
\[
  \sum_{s=1}^S O\Big(\frac{m}{2^s} d^6\ln^2d  \Big) = O\big(m d^6\ln^2d \big) = O \big(d^7\ln^2d \big).
  \qedhere
\]
\end{proof}

\section{Reduction to vector balancing}\label{BDGL}

In this section we describe an alternative approach for proving asymptotic estimates for colorful vector balancing constants matching 
Theorems~\ref{EuclideanThm} and \ref{InfinityThm}, based on the  proof techniques of Lovász, Spencer and Vesztergombi \cite{LSV86} and Bansal, Dadush, Garg, and Lovett \cite{algban}. For the following proof we denote the colorful vector balancing constant of two symmetric convex bodies $K,L\subseteq \mathbb{R}^d$ as
\begin{equation*}
    \mathrm{colvb}(K,L):=\sup_{n\geq d}\ \max_{\stackrel{V_1,...,V_n\subseteq K}{0\in \sum \conv V_i}}\ \min_{v_i \in V_i,\ i\in[n]}\|v_1+\cdots+v_n\|_L.
\end{equation*}

\begin{theorem}\label{BDGLred}
    Given any symmetric convex bodies $K,L\subseteq \mathbb{R}^d$,
    \begin{equation*}
        \mathrm{colvb}(K,L)\leq 2\mathrm{vb}(K,L).
    \end{equation*}
\end{theorem}
\begin{proof}
    We are given families $V_1,...,V_n\subseteq K$ and a vector $\lambda\in \Delta_V$ so that $V\lambda=0$. Note that by Carathéodory's theorem we may assume that $|V_i|\leq d+1$ for each $i\in[n]$. Let
    \[
    \rho=\max_{u \in K} \|u\|_L.
    \]
    Fix $\varepsilon>0$     and take $\ell\in \mathbb{Z}$ so that $n(d+1)2^{-(\ell-1)} \rho\leq \varepsilon$.

    Each coordinate $\lambda(v)$ of $\lambda$, for $v\in V_i, i\in[n]$, has a binary expansion, which we truncate at the $\ell^{th}$ digit after the radix point to obtain the vector $\mu$ with coordinates $\mu(v)$ so that $|\mu(v)-\lambda(v)|\leq 2^{-(\ell-1)}$ for each $v\in V_i$, $i\in[n]$. Then
\begin{equation}
    \|V\lambda-V\mu\|_L=\Big\|\sum_{i\in[n]}\sum_{v\in V_i}(\lambda(v)-\mu(v))v\Big\|_L\leq \frac{1}{2^{\ell-1}}\sum_{i\in[n]}\sum_{v\in V_i}\|v\|_L\leq \frac{n(d+1)\rho}{2^{\ell-1}}\leq \varepsilon.
\end{equation}

Denote the $j^{th}$ digit of the binary expansion of $\mu(v)$ by $\mu(v)^{(j)}$. We define the set
    \begin{equation*}
        S_\ell:=\{v\in \cup_{i\in[n]}V_i:\ \mu(v)^{(\ell)}=1\}
    \end{equation*}
    to be the set of vectors in our collection for which the $\ell^{th}$ digit of the binary expansion of the corresponding coefficient is $1$. Since $\sum_{v\in V_i}\lambda(v)=1$ for each $i\in[n]$, it follows that $|S_\ell\cap V_i|=2q_i$ for some $q_i\in \mathbb{Z}$, so we can write $S\cap V_i=\{v_1^i,...,v_{2q_i}^i\}$ for each $i\in[n]$. We define the auxiliary collections of vectors
    \begin{equation*}
        W_i=\Big\{\tfrac{v^i_{2j}-v^i_{2j-1}}{2}\Big\}_{j\in[q_i]}\subseteq K
    \end{equation*}
    and then balance the collection $W=\cup_{i\in[n]}W_i$, yielding signs $\chi_i(j)\in \{\pm 1\}$ so that
    \begin{equation*}
        \left\|\sum_{i\in [n]}\sum_{j\in[q_i]}\chi_i(j)\frac{v^i_{2j}-v^i_{2j-1}}{2}\right\|_L\leq \mathrm{vb}(K,L).
    \end{equation*}

Color the elements of $S_\ell$ as follows: for each $i\in[n]$, for each $k\in[2q_i]$, we assign $\beta_i(k)=\chi_i(j)$ for $k$ even and $\beta_i(k)=-\chi_i(j)$ for $k$ odd, so that
\begin{equation}
\left\|\sum_{i\in [n]}\sum_{k\in[2q_i]}\beta_i(k)v^i_k\right\|_L=2\left\|\sum_{i\in [n]}\sum_{j\in[q_i]}\chi_i(j)\frac{v^i_{2j}-v^i_{2j-1}}{2}\right\|_L\leq 2\mathrm{vb}(K,L).
\end{equation}
We then update the vector $\mu$ as follows: for $v\not\in S_{\ell}$, $\mu_1(v)=\mu(v)$. For $v\in S_\ell$, we know that $v=v^i_k$ for some $i\in[n],k\in[2q_i]$, and we update $    \mu_1(v):=\mu(v)+2^{-\ell}\beta_i(k)$. By construction, $\mu_1\in \Delta_V\cap 2^{-(\ell-1)}$, and $\|V\mu_1\|\leq 2^{-(\ell-1)}\mathrm{vb}(K,L)$. Iterating the argument for the successive digits leads to a selection vector $\mu_\ell$ for which
\begin{equation*}
   \left\|V\lambda\right\|_L\leq \|V\lambda-V\mu\|_L+\|V\mu\|_L\leq\left\|\sum_{i=0}^{\ell-1}2^{i}\mathrm{vb(K,L)}\right\|_L\leq \varepsilon+2\mathrm{vb}(K,L).
\end{equation*}
As $\varepsilon>0$ was arbitrary, the theorem follows.
 \end{proof}

Combining Theorem \ref{BDGLred} with \eqref{vb_l2} and \eqref{vb_linfty} implies our Theorems \ref{EuclideanThm} and \ref{InfinityThm} up to constants. We intended to give a direct proof that is more suited to algorithmic applications. Indeed, the computational complexity of finding a solution by the above approach depends heavily on the number of vector families $n$, whereas our technique illuminates the geometric aspects of the problem and the independence of the number of vector families, including the reduction to $O(d)$ total vectors that is necessary in the maximum norm case. Moreover, it leads to the sharp bound of $\sqrt{d}$ for the Euclidean case as opposed to the asymptotic bound above, and it improves on the constant for the maximum norm given by combining \cite{LM15} with Theorem \ref{BDGL}.

\section{Acknowledgements}
The authors thank Imre Bárány and Thomas Rothvoss for their helpful suggestions and the anonymous referee for the helpful comments and  pointing out the argument sketched in Section~\ref{BDGL}.

 \bibliographystyle{abbrv}
\bibliography{ColorfulVBBib}

\bigskip

\noindent
{\sc Gergely Ambrus}
\smallskip

\noindent
{\small 
{\em Bolyai Institute, University of Szeged, Hungary \\ and\\ 
 Alfréd Rényi Institute of Mathematics, Budapest, Hungary
 }}
\smallskip

\noindent
e-mail address: \texttt{ambrus@renyi.hu}

\bigskip

\noindent
{\sc Rainie Bozzai}
\smallskip

\noindent
{\small{\em University of Washington, Seattle \\and\\ Bolyai Institue, University of Szeged, Hungary}}
\smallskip

\noindent
e-mail address: \texttt{lheck98@uw.edu}
\end{document}